\documentclass[12pt, a4paper]{amsart}
\usepackage{enumerate, longtable}
\usepackage{amsmath, amscd, amsfonts, amsthm, amssymb, latexsym, comment, stmaryrd, graphicx}
\usepackage{xcolor}
\usepackage[all]{xy}
\usepackage{a4wide}
\usepackage[
        colorlinks,
        backref,
]{hyperref}

\theoremstyle{plain}
\newtheorem{theorem}{Theorem}[section]
\newtheorem{proposition}[theorem]{Proposition}
\newtheorem{lemma}[theorem]{Lemma}

\newtheorem{corollary}[theorem]{Corollary}

\theoremstyle{definition}

\newtheorem{remark}[theorem]{Remark}

\numberwithin{equation}{section}

\newcommand{\bfa}{\mathbf{a}}
\newcommand{\bft}{\mathbf{t}}

\newcommand{\bfT}{\mathbf{T}}

\newcommand{\Gal}{{\rm Gal}}

\newcommand{\ZZ}{\mathbb{Z}}

\newcommand{\QQ}{\mathbb{Q}}

\newcommand{\FF}{\mathbb{F}}

\newcommand{\cP}{\mathcal{P}}

\newcommand{\geom}{{\rm geom}}

\renewcommand{\deg}{\mathrm{deg}}

\begin{document}

\title{Explicit Hilbert's Irreducibility Theorem in Function Fields}
 \dedicatory{Dedicated to the memory of Wulf-Dieter Geyer, 1939-2019.}
\date{\today}

\author{Lior Bary-Soroker}

\address{Raymond and Beverly Sackler School of Mathematical Sciences, Tel Aviv University, Tel Aviv 69978, Israel}
\email{barylior@tauex.tau.ac.il}

\author{Alexei Entin}

\address{Raymond and Beverly Sackler School of Mathematical Sciences, Tel Aviv University, Tel Aviv 69978, Israel}
\email{aentin@tauex.tau.ac.il}

\begin{abstract}
	We prove a quantitative version of Hilbert's irreducibility theorem for function fields: If $f(T_1,\ldots, T_n,X)$ is an irreducible polynomial over the field of rational functions over a finite field $\mathbb{F}_q$ of characteristic $p$, then the proportion of $n$-tuples $(t_1,\ldots, t_n)$ of monic polynomials of degree $d$ for which $f(t_1,\ldots, t_n,X)$ is reducible out of all $n$-tuples of degree $d$ monic polynomials is $O(dq^{-d/2})$.
\end{abstract}

\maketitle

\section{Introduction}

Hilbert's irreducibility theorem says that if $f(T_1,\ldots, T_n,X)\in \QQ[T,X]$ is a multivariate  irreducible polynomial of positive degree in $X$, then the set 
\[
	H_{f} = \{ (t_1,\ldots, t_n)\in \QQ^n : f(t_1,\ldots, t_n,X) \in \QQ[X] \textnormal{ is irreducible}\}\subseteq  \QQ^n
\] 
is non-empty. We call such a set a (basic) \emph{Hilbert set}.

The original motivation of Hilbert was in the inverse Galois problem, see for example \cite{Serre}.  Since then the theorem found numerous applications to different problems in algebra and number theory; e.g., for constructing elliptic curves of high rank, see \cite{Serre2}.

There are several types of qualitative results that indicate that any Hilbert set $H_f$ is big (see \cite[\S 13]{FJ} for reference): $H_f$ is infinite; $H_f$ is Zariski dense; $H_f$ is $S$-adic dense,  there exist many finite sets of primes $S$ for which $H_f$ is $S$-adic open.  

A more quantitative variant is by counting:
for $N\geq 1$, set 
\[
	H_f(N) = \#( H_f \cap \{ (t_1,\ldots, t_n)\in \ZZ^n : N\leq t_i< 2 N ,\ i=1,\ldots, n\})
\]
to be the number integral tuples in $H$ with entries in $[N,2N)$. 
S.~D.~Cohen \cite{Cohen} proved that 
\begin{equation}\label{eq:SDCohen}
	\frac{H_f(N)}{N^n} =  1+O_f\Big(\frac{\log N}{\sqrt{N}}\Big), \qquad N\to \infty.
\end{equation}
Here the notation $O_f$ means that the implied constant may depend on $f$, but not on $N$.
The polynomial $f(T,X)=X^2-(T_1+T_2+\cdots +T_n)$ shows that the error term is optimal (up to the logarithmic factor). 

The qualitative results mentioned above were proved also in the global function field setting. However to the best of our knowledge, \eqref{eq:SDCohen} has not been proved in function fields. An additional challenge arising in this setting is having to deal with inseparable polynomials, a phenomenon that does not occur in number fields. 

 The goal of this work is to prove a function field analogue. We denote by $\FF_q$ the finite field of $q$ elements. We let $\FF_q(u)$ be the field of rational functions in $u$ with coefficients in $\FF_q$. The set $M_{d}\subseteq \FF_q(u)$ of monic polynomials of degree $d$ takes the role of the interval $[N,2N)$ with $q^d=\#M_d$ being the analog of $N = \#[N,2N)\cap \ZZ$. 

\begin{theorem}\label{thm:main}
	Let $K = \FF_q(u)$ be of characteristic $p$, let $f(T_1,\ldots, T_n,X) \in K[T_1,\ldots, T_n,X]$ be an irreducible polynomial of positive degree in $X$. Let $H_f(d)$ be the number of tuples $(t_1(u),\ldots, t_n(u))\in M_d^n$ with $f(t_1,\ldots, t_n,X) \in K[X]$ irreducible. Then, 
	\begin{equation}\label{eq:main}
		\frac{H_f(d)}{q^{dn}} = 1+O_{f,q}(dq^{-d/2}),
	\end{equation}
	where the implicit constant may depend on $f,q$.
\end{theorem}

To compare with \eqref{eq:SDCohen} we note that $q^d$ corresponds to $N$ and $d$ to $\log N$, so the error term in \eqref{eq:main} is the perfect analog of \eqref{eq:SDCohen}.
\begin{remark} In fact we will show that if one fixes $\deg_{\bfT,X}f$ and a constant $\lambda$ then a similar bound holds uniformly whenever $\deg_uf\le\lambda d$. See Theorem \ref{thm:nonsep} for the precise statement.\end{remark}

When $f$ is separable in the variable $X$, we will prove \eqref{eq:main} using  the function-field large sieve inequality due to Hsu \cite{Hsu96}, a proof that goes in the same lines of the proof of \eqref{eq:SDCohen}. When $f$ is inseparable, those ideas seem to fail. Instead we use an approach of Uchida \cite{Uchida}.

Theorem~\ref{thm:main} has many consequences. The first is to Galois theory, which states that if $L/K(T_1,\ldots, T_n)$ is a Galois extension, then for most $(t_1,\ldots, t_n)\in M_d^n$, the Galois group remains the same under the specialization $T_i\mapsto t_i$. See Corollary~\ref{cor:sep} for a precise statement. 

Theorem 1.1 gives a new proof to the following assertions, that also can be derived from \cite[Thm 13.3.5]{FJ}.
\begin{corollary}
	Let $f(T_1,\ldots, T_n,X)\in K[T_1,\ldots, T_n,X]$ be irreducible of positive degree in $X$ and $H_f$ the set of $(t_1,\ldots, t_n) \in \FF_q[u]^n$ for which $f(t_1,\ldots,t_n ,X)\in \FF_q(u)[X]$ is irreducible. 
	\begin{enumerate}
		\item $H_f$ contains a Zariski dense set of prime $n$-tuples; i.e., for every nonzero $g(T_1,\ldots, T_n)$, there exists $(t_1,\ldots,t_n) \in H_f$ such that $g(t_1,\ldots,t_n )\neq 0$ and $t_i\in \FF_q[u]$ is irreducible for all $i$.
		\item If $S$ is a finite set of primes of $\FF_q(u)$, then $H_f$ is $S$-adically dense. 
	\end{enumerate}
\end{corollary}

The implication of the corollary from Theorem~\ref{thm:main} is immediate: the primes have density bigger than  $q^{-\epsilon}$, fo any $\epsilon$ and any $S$-adic open set has positive density; hence those sets cannot be contained in the complement of $H_f$ by Theorem~\ref{thm:main}.

\subsection*{Acknowledgments}
The authors thank Dan Haran for some illuminating discussions of $p$-th powers in characteristic $p$ and Arno Fehm for helpful remarks that improved the presentation of the paper.

The first named author was partially supported by grant no.\ 702/19 of the Israel Science Foundation and the second by grant no.\ 2507/19 of the Israel Science Foundation.

\section{Set-up and preliminaries}
Let $\FF_q$ be a finite field with $q$ elements and of characteristic $p$, let $K=\FF_q(u)$ be the field of rational functions in $u$. We denote by $M_d$ the subset of $K$ of monic polynomials of degree $d$. We denote tuples by bolded letters, e.g.\ we write $\bft$ for $(t_1,\ldots, t_n)$ and $\bfT$ for $(T_1,\ldots,T_n)$.

To a polynomial $f\in K[\bfT,X]$ in $n+1$ variables over the field $K$ we attach the following data. 
Let $\deg_\bfT f$ (resp. $\deg_{\bfT,X}f$) denote the total degree of $f$ in the variables $T_1,\ldots,T_n$ (resp. $T_1,\ldots,T_n,X$) and denote $\deg_u(f)=\max\left(\deg_u\tilde{f},\deg_ua\right)$, where $f=\tilde{f}/a,\tilde{f}\in\FF_q[u,\bfT,X],a\in\FF_q[u]$ is a reduced fraction.

It is convenient to work with polynomials in $\FF_q[u,\bfT,X]$ that are monic in $X$. The following lemmas allow us to reduce to this case.
\begin{lemma}\label{lem:monic}
	Let $f\in K[\bfT,X]$ be a polynomial, let $a(u)\in \FF_q[u]$ be the common denominator of the coefficients of $f$, $\tilde{f}=af\in\FF_q[u,\bfT,X]$, $b\in \FF_q[u,\bfT]$ the leading coefficient of $\tilde{f}$ (in the variable $X$) and $r=\deg_X f$. Then the polynomial 
	\[
		g (\bfT,X) =  b^{\deg_Xf-1} \tilde{f}\Big(\bfT, \frac{X}{b}\Big) \in \FF_q[u,\bfT,X]
	\]
	satisfies:
	\begin{enumerate}
		\item $g$ is monic in $X$.
		\item $\deg_X g=\deg_X f$, $\deg_\bfT g\le\deg_\bfT f\cdot\deg_Xf$ and $\deg_ug\le\deg_u f\cdot\deg_X f$
		\item For all $\bft \in M_d^n$ with $b(\bft)\neq 0$ each of the polynomials $g(\bft,X)$ and $f(\bft,X)$ is irreducible iff the other is. 		
	\end{enumerate}
\end{lemma}

\begin{proof} Straightforward from the definition of $g$.\end{proof}

The following lemma will be useful in several instances.

\begin{lemma}\label{lem:bound}
	Let $h(\bfT) \in K[T_1,\ldots, T_n]$ be a nonzero polynomial. Then the number of tuples $\bft\in M_{d}^n$ such that $h(\bft)=0$ is at most $\sum_j \deg_{T_j} h \cdot q^{d(n-1)}$. 
\end{lemma}

\begin{proof}
	We apply induction. If $n=1$, then $h(T)$ has at most $\deg_T h$ solutions in $K$, hence in $M_{d}$. Assume $n\geq 2$ and write $h(\bfT) = \sum_{i} h_i(T_1,\ldots, T_{n-1})T_n^i$, where not all of the $h_i$-s are zero. 
	Then, by induction, 
	\[
		\#\{ \bft\in M_{d}^n:h(\bft)=0\} \leq  \sum_{\substack{\bft'\in M_{d}^{n-1} \\ \forall i:h_i(\bft') =0}} q^d + \sum_{\substack{\bft'\in M_{d}^{n-1} \\ \exists i:h_i(\bft') \neq 0}} \deg_{T_n} h \le \sum_j\deg_{T_j}h\cdot q^{n-1},
	\]
	as needed.
\end{proof}

The last two lemmas immediately imply the following result which allows us to reduce to monic polynomials with coefficients in $\FF_q[u,\bfT]$:

\begin{corollary}\label{cor:reduction}
	Let $f\in K[\bfT,X]$ be a polynomial and  $g(\bfT,X)\in \FF_q[u,\bfT,X]$ as in Lemma~\ref{lem:bound}. Then for all $\bft\in M_{d}^n$ but at most $\sum\deg_{T_j}h\cdot q^{d(n-1)}$, each of $f(\bft,X)$ and $g(\bft,X)$ is irreducible iff the other is.
\end{corollary}

\section{Separable Polynomials}

In this section we assume that 
$f\in \FF_q[u,\bfT,X]$ is irreducible and that it is separable and monic in $X$. Throughout this and the next section we fix a bound on $\deg_{\bfT,X}f$ and all asymptotic notation will have implied constants that may depend on this bound. We do not fix the degree in $u$ as we are looking for a result which is uniform in $f,q$ as long as $\deg_uf$ grows at most linearly with $d$.

Let $L$ be the splitting field of $f$ over $K(\bfT)$ and let $\FF_{q^r} = \overline{\FF}_q\cap L$ be the algebraic closure of $\FF_q$ in $L$. We have the following diagram of fields:

\[
	\xymatrix{
		& L \ar@{-}[r]&L\overline{\FF}_q
		\\
		K(\bfT)\ar@{-}[r] & K \FF_{q^r}(\bfT)\ar@{-}[r]\ar@{-}[u]& K\overline{\FF}_q (\bfT)\ar@{-}[u]
		\\
		\FF_q\ar@{-}[r] \ar@{-}[u]& \FF_{q^r}\ar@{-}[r]\ar@{-}[u] & \overline{\FF}_q\ar@{-}[u]
	}
\]

We denote  $G=\Gal(L/K(\bfT))$ and  $G_{\geom}=\Gal(L/K\FF_{q^r}(\bfT)) \cong \Gal(L\overline{\FF}_q/K\overline{\FF}_q(\bfT))$. We have the fundamental exact sequence 
\begin{equation}\label{eq:fund_seq}
	\xymatrix@1{1\ar[r]&G_{\geom}\ar[r]&G\ar[r]& \ZZ/ r\ZZ\ar[r] &1}
\end{equation}
where $\ZZ/r\ZZ \cong\Gal(\FF_{q^r}/\FF_q)$ and under this identification the map $G\to \ZZ/r\ZZ$ is the restriction-of-automorphism map. 

\begin{remark}\label{remark:r}By a result of Guralnick \cite{Guralnick} any cyclic quotient of a transitive permutation group of degree $N$ has order $\le N$ and therefore since $\ZZ/r\ZZ$ is a quotient of $G$ which acts transitively on the roots of $f$ we must have $r\le\deg_Xf$.\end{remark}

If $\bft \in K^n$ is an $n$-tuple for which $f(\bft,X)$ is separable in $X$, then we denote by $L_{\bft}$ the splitting field of $f(\bft,X)$ over $K$ and we note that $\FF_{q^r}\subseteq L_{\bft}$. We further denote  $G_{\bft} = \Gal(L_{\bft}/K)$ and  $G_{\geom,\bft} = \Gal(L_{\bft}/K\FF_{q^r})$. We have that $G_{\bft}\leq G$ and $G_{\geom,\bft}\leq G_{\geom}$, in a canonical way up to conjugation, and $G/G_{\geom} \cong G_{\bft}/G_{\geom,\bft} \cong \ZZ/r \ZZ$. Hence the following observation is immediate:
\begin{lemma}\label{lem:GGg}
	$G_{\bft} \cong G$ if and only if $G_{\geom,\bft}\cong G_{\geom}$.
\end{lemma}

Let $P(u)\in \FF_{q}[u]$ be a prime polynomial; that is to say, irreducible and monic. Write $|P|=q^{\deg P}$ for the norm of $P$. 
We may reduce the coefficients of $f$ modulo $P$, and we denote the resulting polynomial by $f_P\in \FF_{|P|}[\bfT,X]$. If in addition $f_P$ is separable  in $X$, then we define $L_P$, $G_{P}$ and $G_{\geom,P}$ in a similar fashion to the definition of $L$, $G$ and $G_{\geom}$ (with $\FF_{|P|}$ taking the role of $K$). Since $f$ is monic in $X$, we have that  $\deg_X f=\deg_X f_P$, so $G_{P}$ and $G_{\geom, P}$ embed in $G$ and $G_{\geom}$, respectively, and this embedding is canonical, up to conjugation in $G$. 

Consider the set of primes of good reduction in the following sense:
\[
	\cP_f = \{ P\in \FF_q[u] : P \textnormal{ prime},\ \ r\mid \deg P,\ f_P \textnormal{ separable in $X$}, \ G_{\geom} \cong G_{\geom, P}\}.
\]
If $P\in \cP_f$, then $\FF_{q^r}\subseteq \FF_{|P|}$, so $L_P$ is regular over $\FF_{|P|}$, and we have the diagram of fields:
\[
	\xymatrix{
		 L_P \ar@{-}[r]&L_P\overline{\FF}_q
		\\
		\FF_{|P|}(\bfT)\ar@{-}[r]\ar@{-}[u]^{G_{\geom}}& \overline{\FF}_q(\bfT) \ar@{-}[u]^{G_{\geom}}
		\\
		\FF_{|P|}\ar@{-}[r] \ar@{-}[u]& \overline{\FF}_q\ar@{-}[u]
	}
\]

\begin{lemma}\label{lem:primes}
	Let $\lambda>0$ be a constant. The number of (monic) irreducible polynomials $P\in\FF_q[u]$ with $\deg P\ge \frac 1\lambda\deg_uf$, $r|\deg P$ and such that $P\not\in\cP_f$ is $O_{\lambda,\deg_{\bfT,X}f}(1)$.
\end{lemma}

\begin{proof}
	We fix $\lambda>0$. Since $f$ is separable in $X$, there is a nonozero coefficient $g\in \FF_q[u,\bfT]$ of $X^i$ for some $i\not\equiv 0\mod p$. If $r|\deg P$ and $f_P$ is not separable, then $P$ must divide all the coefficients of $g$. Since we assume $\lambda \deg P\ge \deg_u f \geq \deg_ug$,  there are at most $O_\lambda(1)$ such $P$ (since the same coefficient of $g$ is divisible by all of them).
	
	It remains to show that if $r|\deg P\ge \deg_uf/\lambda$ and $f_P$ is separable we have $G_{\geom} \cong G_{\geom, P}$ with $O_{\lambda,\deg_{X,\bfT} f}(1)$ many exceptions. For this we choose a monic irreducible polynomial $h\in K\FF_{q^r}[\bfT,X]$ whose root generates $L$ ($h$ is a resolvent of $f$). We can choose $h$ with $\deg_u h=O(\deg_uf)$, $\deg_X h = |G_{\geom}|$ and $h$ is monic in $X$. The resolvent $h$ is geometrically irreducible. If $h_P$ is geometrically irreducible then $G_{\geom, P } \cong G_{\geom}$. 
	
	Let $c_1,\ldots,c_m$ be the coefficients of $h$. By a theorem of Noether (see the beautiful exposition by Geyer \cite[Theorem 5.3.1]{Geyer}) there exist $O_{\deg_{\bfT,X} f}(1)$ polynomials $F_j\in\ZZ[v_1,\ldots,v_m]$ of degree $O(1)$ such that $F_j(c_1,\ldots,c_m)\neq 0$ and $h_P$ is absolutely irreducible unless $F_j(c_1,\ldots,c_m)\equiv 0\pmod P$ for some $j$. Since $\deg_uF_j(c_1,\ldots,c_m)=O(\deg_uf)$ there are only $O(1)$ such $P$ with $\deg P\ge \deg_uf/\lambda$ if we assume $\deg_uf\le \lambda d$. For any other $P$ we have $G_{\geom, P } \cong G_{\geom}$.
\end{proof}

Take $P\in \cP_f$. Every $\bfa \in \FF_{|P|}^n$ with $f_P(\bfa,X)\in \FF_{|P|}[X]$  separable  gives rise to the Frobenius conjugacy class 
\[
	C_{\bfa} \subseteq G_{P} = G_{\geom,P} \cong G_{\geom}.
\]
Now fix a conjugacy class $C\subseteq G_{\geom}$ and let 
\[
	\Omega_P =\Omega_P(C) = \{ \bfa \in \FF_{|P|}^n : C_{\bfa} \neq C\}. 
\]
(The set $\Omega_P$ also contains those $\bfa$ for which $f_P(\bfa, X)$ is not separable.) 
The explicit Chebotarev density theorem for function fields (see Theorem 3 in \cite{Entin}, for a version with the desired uniformity) gives that 
\[
	\frac{\# \Omega_{P}}{|P|^n} = \Big(1- \frac{|C|}{|G_{\geom}|}\Big) \cdot \big(1+O_{\deg_{\bfT,X} f}\big(|P|^{-1/2}\big)\big).
\]
In particular, if $G_{\geom}$ is non-trivial, then $|C|\neq |G_{\geom}|$, so if we take any  $\frac{|C|}{|G|}<c<1$, then for every $P\in \cP_f$ of sufficiently large degree, we have 
\begin{equation}\label{eq:OmegaPbds}
	\frac{\# \Omega_{P}}{|P|^n}\leq c.
\end{equation}

\begin{proposition} \label{prop:sep}
	Fix a constant $\lambda>0$ and assume $\deg_u f\leq \lambda d$. Let $B_C(d)$ be the number of tuples $\bft \in M_d^n$ such that $G_{\geom,\bft} \cap C\neq \emptyset$. Then  we have
	\[
		\frac{B_C(d)}{q^{dn}} =  O_{\deg_{\bfT,X} f,\lambda}\left(q^{-d/2+r} d\right).
	\]
with $r\leq \deg_Xf$ as defined above.
\end{proposition}

\begin{proof}
	Let $\Omega = \Omega(d,C,f)$ be the set of tuples $\bft \in M_d^n$ such that $G_{\geom,\bft} \cap C=\emptyset$. 
	Let $\bft\in M_{d}^n$ and $P\in \cP_f$. Since $r\mid \deg P$, we have that $P$ is totally split in $K\FF_{q^r}$. Moreover, $f(\bft,X)\mod P = f_P(\bfa,X)$, where $\bfa:= \bft \mod P$. 
	If we assume that $f_P(\bfa,X)$ is separable, then we get that the Frobenius class of $P$ in $L_{\bft}$ equals the Frobenius class of $\bfa$ in $L_{P}$, which we earlier denoted by $C_{\bfa}$. Hence if $\bft\in \Omega$, then $C_{\bfa}\neq C$ and so 
	\[
		\bft\mod P\in \Omega_P
	\]
	for any $P\in \mathcal{P}_f$.
	
	We apply the large sieve inequality for function fields \cite[Theorem 3.2]{Hsu96}	
	\footnote{In the notation of \cite[Theorem 3.2]{Hsu96} we take $X=\Omega, N=d-1, f=(t^d,\ldots,t^d),K=d/2$, $\alpha_P=c$ if $P\in\mathcal{P}_f$ and $\alpha_P=1$ otherwise.} to conclude that 
	\[
		\# \Omega \leq  \frac{q^{dn}}{L},
	\]
	where 
	\[
		L \geq 1+\sum_{\substack{P\in \cP_f\\ \deg P\leq d/2}} \frac{1-c}{c} \gg \# \{P\in \cP_f: \deg P\leq d/2\}.
	\]
	Finally, let $s$ to be an integer so that  $d/2-r \le rs \leq d/2$ (if it is not positive the assertion of the proposition is trivial). If $d$ is large, by Lemma~\ref{lem:primes}, $\cP_f$ contains all but $O(1)$ primes of degree $rs$, so  
	\[
		\# \{ P \in \cP_f : \deg P\leq d/2\} \geq \pi_{q}(rs)-O(1)\sim \frac{q^{rs}}{rs} \gg\frac{q^{d/2-r}}{d/2}.
	\]
	Here $\pi_q(rs)$ denotes the number of primes of degree $rs$ and we use the Prime Polynomial Theorem (see e.g.\ \cite[Theorem~2.2]{Rosen}) which says that $\pi_q(rs) = \frac{q^{rs}}{rs} + O(q^{-rs/2})$.
	We conclude that 
	\[
		\frac{\#\Omega}{q^{dn}} \ll q^{-d/2+r}d. 
	\]
	Since $B_{C}(d)=\#\Omega$ the proof is done.
\end{proof}

\begin{corollary}\label{cor:sep}
	Fix $\lambda>0$  and assume $\deg_uf\le\lambda d$. Let $B(d)$ be the number of $n$-tuples $\bft \in M_{d}^n$ such that $G_{\bft} \not\cong G$. Then
	\[
		\frac{B(d)}{q^{dn}} =  O_{\lambda,\deg_{\bfT,X}f} (q^{-d/2+r}d),
	\]
	with $r\leq \deg_X f$ as defined above.
\end{corollary}

\begin{proof}
	By Lemma~\ref{lem:GGg}, if $G_{\bft}\not \cong G$, then $G_{\geom,\bft}\not \cong G_{\geom}$. If $G_{\geom,\bft}$ is isomorphic to a proper subgroup of $G_{\geom}$, then there exists a conjugacy class $C$ such that $G_{\geom,\bft}\cap C \neq \emptyset$ (it is an elementary exercise in group theory to show that every proper subgroup of a finite group is disjoint from some conjugacy class).
	Applying Proposition~\ref{prop:sep}, the union bound, and the fact that there are $O_{\deg_{X}f}(1)$ conjugacy classes in $G_{\geom}$, we conclude that
	\[
		B(d)  \ll_{\deg_{\bfT,X}f,\lambda} q^{-d/2+r}d,
	\]
	as needed.
\end{proof}

\begin{theorem}\label{thm:uniform}
	Let $K=\FF_{q}(u)$, let $f(T_1,\ldots, T_n,X)\in K[T_1,\ldots,T_n,X]$ be an irreducible polynomial that is separable and of positive degree in $X$. Let $H(d)$ be the number of tuples $\bft \in M_{d}^n$ with $f(\bft,X)$ irreducible, separable, and $\deg_X f(\bft,X)=\deg_X f(\bfT,X)$. Let $\lambda>0$ be a constant and assume that $\deg_uf\le\lambda d$. Then 
	\[
		\frac{H(d)}{q^{dn}} = 1+O_{\deg_{T_1,\ldots,T_n,X}f,\lambda} \left(dq^{-d/2+r}\right),
	\]
	with $r\leq \deg_X f$ as defined above.
\end{theorem}

\begin{proof}

We fix $\lambda$ and a bound on $\deg_{\bfT,X}f$. First we observe that by Remark \ref{remark:r} we have $r\le\deg_Xf$.
	By Corollary~\ref{cor:reduction} we may assume w.l.o.g.\ that $f\in \FF_q[u,\bfT,X]$ and that $f$ is monic in $X$. 

	Let $\Sigma$ be the set of tuples $\bft \in M_{d}^n$ for which $f(\bft,X)$ is inseparable; i.e., the zero set of all the coefficients of $X^i$ for $i\not\equiv 0\mod p$. 
	By Lemma~\ref{lem:bound},  $\#\Sigma = O(q^{d(n-1)})$. For $\bft\not\in \Sigma$, the action of $G_{\bft}$ on the roots of $f(\bft,X)$ coincides with the action the decomposition group $G_{\bft}\cong D_{\bft}\leq G$ on the roots of $f(\bfT,X)$, hence if $G_{\bft}\cong G$, the action is transitive, and $f(\bft,X)$ is irreducible. 
	Corollary~\ref{cor:sep} then completes the proof. 
\end{proof}

\section{Inseparable polynomials}

In this part, we let $f(\bfT,X)\in K [T_1,\ldots, T_n,X]$ be an irreducible polynomial that is inseparable in $X$. 

Our treatment is inspired from Uchida's work \cite{Uchida} and its presentation in \cite{FJ}. 
We will use the following criterion of Uchida for irreducibility of inseparable polynomials, see \cite[Lemma 12.4.1]{FJ}. 
\begin{lemma}\label{lem:irrpthpower}
	Let $F$ be a field of positive characteristic $p>0$, let $g\in F[X]$ be an irreducible monic  polynomial. Assume that at least one of  the coefficients of $g$ is not a $p$-th power in $F$. Then $g(X^{p^\alpha})$ is irreducible for all $\alpha>0$. 
\end{lemma}

We shall need a technical lemma.

\begin{lemma}\label{lem:pthpower}
	Let $h(\bfT) \in \FF_q[u,\bfT]$ be a polynomial which is not a $p$-th power. Let
	$A(d)$ be the number of $n$-tuples $\bft \in M_d$ such that $h(\bft)$ is a $p$-th power. Fix a constant $\lambda>0$. Then as long as $\deg_u h\le\lambda d$ we have 
	\[
		\frac{A(d)}{q^{dn}} \ll_{\deg_\bfT h,\lambda} q^{-d/2+1}.
	\]
\end{lemma}

\begin{proof}
	We fix $\lambda>0$ and a bound on $\deg_\bfT h$. Since every element in $\FF_q$ is a $p$-th power and since $h$ is not a $p$-th power, $h$ is not a polynomial in $v^p$ for at least one of the variables $v\in \{u,T_1,\ldots, T_n\}$. Equivalently, $\frac{\partial h}{\partial v}\neq 0$. Similarly, for $\bft\in M_{d}^n$, $h(\bft)$ is a $p$-th power if and only if $d (h(\bft))/du=0$.

	If $\frac{\partial h}{\partial T_i}=0$ for all $i$, hence $\frac{\partial h}{\partial u} \neq 0$, then by the chain rule,  for all $\bft \in M_{d}^{n}$ we have 
	\[
		\frac{d (h(\bft))}{d u} =  \frac{\partial h}{\partial u} (\bft).
	\]
	By Lemma~\ref{lem:bound}, 
	\[
		\frac{A(d)}{q^{dn}}  = \frac{\#\{\bft\in M_d^n : \frac{\partial h}{\partial u} (\bft)=0\} }{q^{dn}}\ll q^{-d},
	\]
	and the proof is done. 
	
	Next assume that $\frac{\partial h}{\partial T_i}\neq 0$ for some $i$. We can reduce this case to the univariate case (i.e.\ $n=1$): To ease notation assume $i=n$.  For  $t_1,\ldots, t_{n-1}\in M_d^{n-1}$ consider the polynomial $g(T) = h(t_1,\ldots, t_{n-1},T)\in\FF_q[u,T]$. Since 
	\[
		\frac{\partial g}{\partial T} = \frac{\partial h}{\partial T_n} (t_1,\ldots, t_{n-1},T)
	\]
	we get by Lemma~\ref{lem:bound} that ${\partial g}/{\partial T}\neq 0$ for all $(t_1,\ldots, t_{n-1})\in M_d^{n-1}$ but $O(q^{d(n-2)})$ such $(n-1)$-tuples. So $n$-tuples $\bft\in M_{d}^n$ for which ${\partial g}/{\partial T}=0$, contribute at most $O(q^{d(n-1)})$ to the total number of $n$-tuples  with $h(\bft)$ a $p$-th power.
	Therefore, we may assume that ${\partial g}/{\partial T}\neq 0$ and it suffices to  prove that under this assumption 
	\begin{equation}\label{eq:count}
		\#\{ t\in M_d : g(t) \in K^p\} = O(q^{d/2+1}),
	\end{equation}
	in order to conclude the proof. 
	
	Take $r(u) \in \FF_q[u]$ irreducible of degree $\lfloor \frac d2 \rfloor$ such that $g,\partial g/\partial T\not\equiv 0\pmod r$. Such an $r$ exists since by our assumptions $\deg_ug,\deg_u(\partial g/\partial T)=O(\deg_uh)=O(d)$ and so for all but $O(1)$ irreducible $r$ of degree $\lfloor \frac d2 \rfloor$ one of the coefficients of $g$ (and of $\partial g/\partial T$) is not divisible by $r$. 
	
	Let $\mathcal{A} = \FF_q[u]/r^2$ and let $\mathcal{S}=\{c^p : c\in \mathcal A\}$ be the set of $p$-th powers modulo $r^2$. We have $\#\mathcal{S} = O\left(q^{\lfloor d/2\rfloor}\right)$: Indeed,  $p$-th powers modulo $r^2$ can come either from the non-invertibles $r\mathcal A$ or from the subgroup $\left(\mathcal A^*\right)^p$ of $\mathcal A^*$ and both sets are of size $O\left(q^{\lfloor d/2\rfloor}\right)$.
	
	For $t\in M_d$, if $g(t) \in K^p$, then $g(t) \mod r^2\in \mathcal S$. 
	Since $\partial g/\partial T\not\equiv 0\pmod r$ we have
\[
\#\left\{t\in M_d:\frac{\partial g}{\partial T}(t)\equiv 0\pmod r\right\}=O\left(q^{d-\deg r}\right)=O\left( q^{d/2+1}\right).
\]
	Hence it suffices to show that
\[
\label{noncritical}\#\left\{t\in M_d:g(t)\bmod r^2\in\mathcal S,\frac{\partial g}{\partial T}(t)\not\equiv 0\pmod r\right\}=O\left( q^{d/2+1}\right).
\]
	Fix $s\in \mathcal S$. Since $g\not\equiv 0\pmod r$ there are $O(1)$ solutions to $g(\tau)\equiv s\pmod r$. If $\tau\in\FF_q[u]/r$ is a root of $g\bmod r$ and $\frac{\partial g}{\partial T}(\tau)\not\equiv 0\pmod r$ then by Hensel's lemma, there are $O(1)$ solutions to $g(\tau') \equiv s\pmod {r^2}$ with $\tau'\equiv\tau\pmod r$ and each of them lifts to $q^{d-2\lfloor d/2\rfloor}$ solutions of $g(t)\equiv s\pmod{r^2},t\in M_d$. So the total number of solutions to $g(t)\equiv s\pmod{r^2},\frac{\partial g}{\partial T}(t)\not\equiv 0\pmod{r^2},t\in M_d$ is bounded by $\#\mathcal{S}\cdot q=O(q^{d/2+1})$, as required.
\end{proof}

Now we may deduce the quantitative Hilbert's irreducibility theorem for inseparable polynomials.

\begin{theorem}\label{thm:nonsep}
	Let $f(\bfT,X)\in K [T_1,\ldots, T_n,X]$ be an irreducible polynomial. Let $H_f(d)$ be the number of tuples $\bft\in M_{d}^n$ for which $f(\bft,X)$ is irreducible. Fix a constant $\lambda>0$ and assume $\deg_u f\le\lambda d$. Then we have  
	\[
		\frac{H_f(d)}{q^{dn}} = 1 + O_{\deg_{\bfT,X}f,\lambda}(d q^{-d/2+r}),
	\]
	where $r\leq deg_X(f)$ is the degree of the algebraic closure $\FF_q^r$ of $\FF_q$ in the splitting field of $f$ over $K(\bfT)$.
\end{theorem}

\begin{proof}
	If $f$ is separable in $X$ this is Theorem \ref{thm:uniform}. Hence we assume that $f$ is inseparable in $X$. By Corollary~\ref{cor:reduction}, we may assume w.l.o.g.\ that $f\in \FF_q[u,\bfT,X]$ and that $f$ is monic in $X$. 

	Since $f$ is inseparable, we may write $f(\bfT,X) = g(\bfT,X^{p^{\alpha}})$ for some $g$ that is irreducible, separable in $X$ and of degree $\deg_X g\leq \deg_Xf$ and for some $\alpha>0$. Let $\FF_{q^r}$ be the algebraic closure of $\FF_q$ inside a splitting field of $g$ over $K(\bfT)$. 
	
	Write $g(\bfT,X) = \sum_i g_i(\bfT) X^i$, with $g_i\in \FF_q[u,\bfT]$. Had all the $g_i$ been $p$-th powers, we would have that $f$ is reducible. Hence there exists $g_i$ which is not a $p$-th power. 
	
	If $g_i(\bft)$ is not a $p$-th power and $g(\bft,X)$ is irreducible, then by Lemma~\ref{lem:irrpthpower}, $f(\bft,X)=g(\bft,X^{p^{\alpha}})$ is irreducible. 
	Applying Theorem~\ref{thm:uniform} and Lemma~\ref{lem:pthpower}, gives that this happens for all but $O_{\deg_{\bfT,X}f,\lambda}(dq^{d(n-1/2)+r})$ of the $\bft$-s (as long as $\deg_uf\le\lambda d$), hence the proof is complete.  
	
\end{proof}

\end{document}